\documentclass[11pt]{article}
\usepackage[toc,page]{appendix}
\usepackage{stmaryrd}
\usepackage{etoolbox}
\usepackage{pst-node}
\usepackage{tikz-cd} 
\appto\appendix{\addtocontents{toc}{\protect\setcounter{tocdepth}{0}}}

\usepackage{amsmath,amssymb,amsthm}
\usepackage{url}
\usepackage{enumitem}
\usepackage{amsfonts,amsmath, amssymb,latexsym}
\usepackage{mathtools}
\usepackage{multirow,comment}
\def\Z{{\mathbb Z}}
 \DeclareFontFamily{U}{wncy}{}

\def\SL{{\rm SL}}
\def\GL{{\rm GL}}

\def\PGL{{\rm PGL}}

\def\Gal{{\rm Gal}}
\def\Cl{{\rm Cl}}

\def\O{{\mathcal O}}
\def\P{{\mathbb P}}

\def\Aut{{\rm Aut}}

\def\F{{\mathbb F}}

\def\Z{{\mathbb Z}}
\def\P{{\mathbb P}}
\def\F{{\mathbb F}}

\def\CC{{\mathcal C}}

\def\Aut{{\text{Aut}}}
\def\Supp{{\text{Supp}}}
\def\Spec{{\text{Spec}}}

\def \gcd{{\rm gcd}}
\def \det{{\rm det}}
\usepackage{enumitem}

\newcommand*{\ra}{\rightarrow}

\newcommand*{\ol}{\overline}

\usepackage{amsthm}
\newtheorem{theorem}{Theorem}

\newtheorem{lemma}[theorem]{Lemma}

\newtheorem{proposition}[theorem]{Proposition}

\def\pp{{\mathfrak{p}}}
\def\qq{{\mathfrak{Q}}}
\def\Div{{\text{Div}}}
\def\deg{{\text{deg}}}

\newtheorem*{definition}{Definition}
\newtheorem{remark}[theorem]{Remark}

\newcommand*{\Oo}{\mathcal{O}}
\newcommand*{\ok}{{\Oo_{K}}}
\newcommand*{\oll}{{\Oo_{L}}}
\newcommand*{\oks}{{\Oo_{K,S}}}
\newcommand*{\oh}{{\Oo_{H}}}
\newcommand*{\olt}{{\Oo_{L,T}}}

\newcommand*{\oht}{{\Oo_{H,T}}}

\usepackage{amscd,amssymb,amsmath}
\usepackage{enumitem}
\usepackage{relsize}
\usepackage{color}
\usepackage{amsmath}

\usepackage{xcolor}
\usepackage{environ}
\definecolor{darkpurple}{rgb}{0.5, 0.0, 0.5}
\newif\ifshowcontent
\showcontentfalse

\NewEnviron{highlighted}{
    \ifshowcontent
    \noindent
    {\bfseries\color{darkpurple} \textit{Note:} \BODY}
    \par
  \fi
}

\author{Fatemehzahra Janbazi and Fateme Sajadi}
 
\title{A unified finiteness theorem for curves}
\begin{document}
\maketitle
\begin{abstract}
We study the arithmetic of Galois-invariant sets of points on algebraic curves with controlled reduction behavior. Let $C$ be a smooth projective curve with a smooth proper model $\mathcal{C}$ over $\mathcal{O}_{K,S}$. We define $\Omega_n$ as the set of $n$-element subsets of $C(\overline{K})$ that are invariant under $\Gal(\overline{K}/K)$ and such that no two points in the set become identified modulo any prime $\mathfrak{p} \notin S$. Our main result establishes that $\Omega_n$ breaks into finitely many orbits under the action of $\Aut_{\mathcal{O}_{K,S}}(\mathcal{C})$, generalizing finiteness theorems of Birch--Merriman, Siegel, and Faltings.
\end{abstract}

\section{Introduction}
Let $K$ be a number field with ring of integers $\mathcal{O}_K$, and let $S$ be a finite set of primes in $\mathcal{O}_K$. Consider a smooth projective curve $C$ of genus $g$ defined over $K$, and let $\mathcal{C}$ be a smooth and proper model of $C$ over $\mathcal{O}_{K,S}$, with an associated automorphism group $\Aut_{\mathcal{O}_{K,S}}(\mathcal{C})$.

\begin{definition}
Define $\Omega_{n,K}(C;S)$ as follows:
\begin{equation*}
    \Omega_{n,K}(C;S) = \left\{ A \subset C(K) :  \begin{array}{ll} &\#A = n \\ &\#r_{\mathfrak{p}}(A) = n \quad \forall \, \mathfrak{p} \notin S  
                \end{array} \right\},
\end{equation*}
where $r_{\mathfrak{p}}$ denotes reduction at $\mathfrak{p}$. 
\end{definition}

We give a precise definition of the reduction map in Section \ref{reduction rp}. Our first finiteness result concerns the action of the automorphism group on this space. Specifically, we show:

\begin{theorem}\label{K-points Theo} 
The number of orbits in $\Aut_{\oks}(\CC)\backslash \Omega_{n,K}(C;S)$ is finite.
\end{theorem}

We now extend this setting to include rational points over finite extensions of $K$, considering Galois-invariant subsets of $C(\overline{K})$. 

\begin{definition}
Define $\Omega_{n,\overline{K}}(C;S)$ as the set of Galois-invariant $n$-element subsets of $C(\overline{K})$ that remain distinct under reduction at all primes outside $S$:
\begin{equation*}
\Omega_{n,\overline{K}}(C;S) =\left \{ A \subset C(\overline{K}) : 
\begin{array}{lll} 
    &\#A = n \\ 
    &\sigma(A)=A \quad &\forall \sigma \in \Gal\big(\overline{K}/K\big) \\ 
    &\#r_{\mathfrak{p}}(A) = n \quad &\forall \,  \mathfrak{p} \notin S
\end{array} 
\right \}.
\end{equation*}
\end{definition}
Our main theorem establishes the finiteness of this space under the action of the automorphism group. Specifically, we have the following result.

\begin{theorem}\label{Kbar-points Theo}
The $\Aut_{\oks}(\CC)\backslash \Omega_{n,\overline{K}}(C;S)$ is finite.
\end{theorem} 

In the language of arithmetic surfaces, the theorem can be restated as follows:

\begin{theorem} For each integer $n\in \mathbb{N}$, there are finitely many horizontal divisors of degree $n$ of $\CC$ that are \'etale  over $\Spec \, \oks$, up to $\Aut_{\oks} (\CC)$.
\end{theorem} 

\subsection{Special Cases}
In this section, we record several classical finiteness theorems as special cases of the results established above. These results are directly used in the proof of Theorem~\ref{K-points Theo}. Here, we simply identify them as special cases without re-proving them.

\subsection*{Birch-Merriman's Theorem}
Let $C \cong \mathbb{P}^1_K$ with the model $\mathcal{C} = \mathbb{P}^1_{\mathcal{O}_{K,S}}$, and let $n \geq 3$. The automorphism group of $\mathcal{C}$ over $\mathcal{O}_{K,S}$ is  
\begin{equation*} 
    \operatorname{Aut}_{\mathcal{O}_{K,S}}(\mathcal{C}) = \operatorname{PGL}_2(\mathcal{O}_{K,S}). 
\end{equation*}  
The following result of Birch and Merriman~\cite[Theorem~1]{birch1972finiteness} establishes the finiteness of equivalence classes of binary forms under the action of \( \GL_2(\mathcal{O}_{K,S}) \), up to scaling by \(S\)-units.
Two homogeneous polynomials \( f,g\in K[x,y] \) are said to be \((K,S)\)-equivalent if there exist
\[
\lambda\in\mathcal{O}_{K,S}^{\times},\qquad 
\gamma\in\GL_2(\mathcal{O}_{K,S})
\]
such that
\[
\lambda f=\gamma\cdot g .
\]
\begin{theorem}\label{BM in intro} 
  
    If $n \geq 3$, then there are only finitely many $(K, S)$-orbits of binary forms $f \in \mathcal{O}_K[x, y]$ with $\deg(f) = n$ and discriminant  $\Delta(f) \in \mathcal{O}_{K,S}^{\times}$.
\end{theorem}
To explain how Theorem \ref{BM in intro} follows from Theorem \ref{Kbar-points Theo}, we define the following set
\begin{equation*}
T(K, S) = \big\{ f \in \mathcal{O}_{K,S}[x, y] : \deg(f) = n,\,  \Delta(f) \in \mathcal{O}_{K,S}^{\times} \big\}.
\end{equation*}
Each $f \in T(K, S)$ determines a set of roots in $\mathbb{P}^1(\overline{K})$, giving a map
\begin{equation*}
R: T(K, S) \to \Omega_{n, \overline{K}}(C; S).
\end{equation*}
The map is well-defined because $\Delta(f) \in \mathcal{O}_{K,S}^{\times}$ guarantees that the roots remain distinct modulo any prime $\mathfrak{p} \notin S$.

The action of $\operatorname{GL}_2(\mathcal{O}_{K,S})$ on $T(K, S)$ by linear transformations preserves the map, allowing it to descend to the quotient. This induces the following map:
\begin{equation*}
\overline{R}: {\sim}_{(K,S)}  \backslash T(K, S)  \to \operatorname{PGL}_2(\mathcal{O}_{K,S}) \backslash \Omega_{n, \overline{K}}(C; S).
\end{equation*}
The map $\overline{R}$ is injective: if two binary forms $f, g \in T(K, S)$ have the same roots in $\mathbb{P}^1(\overline{K})$ and both have discriminant in $\mathcal{O}_{K,S}^{\times}$, then they are $(K, S)$-equivalent. Since Theorem~\ref{Kbar-points Theo} shows that $\operatorname{PGL}_2(\mathcal{O}_{K,S}) \backslash \Omega_{n, \overline{K}}(C; S)$ is finite, it follows that $T(K, S)$ has finitely many $(K,S)$-equivalence classes.

\subsection*{Siegel's Theorem}
Let $C$ be an elliptic curve over $K$ with a smooth, proper model $\mathcal{C}$ over $\mathcal{O}_{K,S}$, and set $n = 2$. Since the restriction map from the automorphism group of the model to that of the curve is injective and, by the theory of minimal models for curves of genus at least one, also surjective, the two groups coincide. Hence:
\begin{align*}
    \Aut_{\mathcal{O}_{K,S}}(\mathcal{C}) &\cong \Aut_K(C), \\
    [\Aut_K(C) &: C(K)] < \infty.
\end{align*}
By Theorem~\ref{K-points Theo}, the set $\Omega_{2,K}(C;S)$ has finitely many orbits under the action of $\Aut_{\mathcal{O}_{K,S}}(\mathcal{C})$. Since $C(K)$ has finite index in this group, it follows that
\begin{equation*}
   \# \, C(K) \backslash \Omega_{2,K}(C;S) < \infty.
\end{equation*}
Each element of $\Omega_{2,K}(C;S)$ is a pair of $K$-points, and by translating, we may assume it is of the form $\{O, z\}$ for some $z \in C(K)$. The condition that the pair lies in $\Omega_{2,K}(C;S)$ ensures that $z$ does not reduce to the origin modulo any prime $\mathfrak{p} \notin S$, so $z$ defines an $\mathcal{O}_{K,S}$-point of $\mathcal{C} \setminus \overline{\{O\}}$, where $\overline{\{O\}}$ denotes the Zariski closure of $O$ in $\mathcal{C}$. In this way, Theorem~\ref{K-points Theo} is closely related to Siegel’s theorem \cite{Siegel} on integral points.

\subsection*{Faltings's Theorem}
When $n = 1$, the set $\Omega_{1,K}$ coincides with the set of $K$-rational points on $C$. If $C$ has genus at least $2$, its automorphism group is finite. In this case, Theorem~\ref{K-points Theo} gives the finiteness of $C(K)$, in agreement with Faltings’ Theorem \cite{Faltings}.

\subsection{Outline}
We begin by establishing Theorem~\ref{K-points Theo}, which forms the foundation for the proofs of our main theorems. The proof of Theorem~\ref{K-points Theo} relies on several Diophantine inputs: Birch–Merriman theorem in genus $0$, Siegel’s theorem in genus $1$, and Faltings’ theorem for genus at least $2$.

After proving Theorem~\ref{K-points Theo}, we introduce a sufficiently large number field $H$ over which all points in $\Omega_{n,\overline{K}}(C;S)$ become $H$-rational. The key step is then to descend from $H$ to $K$, showing that finiteness of orbits over $H$ implies finiteness over $K$. To carry out this descent, we analyze the natural map
\begin{equation*}
\Aut_{\mathcal{O}_{K,S}}(\CC) \backslash \Omega_{n, \overline{K}}(C; S) \to \Aut_{\mathcal{O}_{H,T}}(\CC_H) \backslash \Omega_{n, H}(C_H; T),
\end{equation*}
where $C_H$ and $\CC_H$ denote the base changes of $C$ and $\CC$ to $H$ and $\mathcal{O}_{H,T}$, respectively, and $T$ is the set of primes of $\mathcal{O}_H$ lying above $S$. Since Theorem~\ref{K-points Theo} guarantees finiteness of orbits over $H$, it remains to show that the fibers of this map are finite.

Since the set we consider is closed under the Galois action, the number of elements in each fiber can be related to the first cohomology group of the Galois action on the stabilizer of a point $A \in \Omega_{n, H}(C_H; T)$. Establishing finiteness of this cohomology group then ensures finiteness over $K$, completing the proof of Theorem~\ref{Kbar-points Theo}.

This paper is organized as follows. Section~\ref{sec 2} provides the necessary background, including the definition of the reduction map, its properties, and a review of models and non-commutative cohomology. In Section~\ref{sec 4}, we prove Theorem~\ref{K-points Theo} by considering different cases based on the genus. Finally, in Section~\ref{sec 3}, we establish Theorem~\ref{Kbar-points Theo} using a Galois descent argument as described above.

\section*{Acknowledgment} 

We thank our advisors, Arul Shankar and Jacob Tsimerman, for proposing the problem that led to this work and for their helpful guidance along the way.

\section{Background and notations}\label{sec 2}

\subsection{Models}

A smooth projective curve $C$ over a field $K$ is a geometrically integral, smooth, and projective scheme of dimension one over $K$. 

\begin{definition}
A \emph{model} of $C$ over $\mathcal{O}_{K,S}$ is a flat, finite type scheme $\mathcal{C} \to \operatorname{Spec}(\mathcal{O}_{K,S})$ whose generic fiber recovers $C$, i.e.,
\[
\mathcal{C} \times_{\mathcal{O}_{K,S}} K \cong C.
\]
A \emph{smooth and proper model} of $C$ over $\mathcal{O}_{K,S}$ is a model $\mathcal{C} \to \operatorname{Spec}(\mathcal{O}_{K,S})$ such that $\mathcal{C}$ is smooth and proper over $\mathcal{O}_{K,S}$.
\end{definition}

For any automorphism $\phi \in \Aut_{\mathcal{O}_{K,S}}(\mathcal{C})$, its restriction to the generic fiber yields an automorphism in the group $\Aut_K(C)$. In the minimal regular model of a curve of genus $g \ge 1$, this restriction is in fact an equality:

\begin{theorem}{\cite[4, Prop.~4.6]{silverman1994advanced}}\label{ext aut}
Let $R$ be a Dedekind domain with fraction field $K$, and let $C$ be a smooth projective curve of genus at least $1$. Let $\mathcal{C}$ be a minimal proper regular model of $C$ over $R$, and let $\mathcal{C}^0$ denote the largest subscheme of $\mathcal{C}$ which is smooth over $R$. Then every $K$-automorphism 
\[
\tau : C \longrightarrow C
\] 
of the generic fiber extends to $R$-automorphisms
\[
\tilde{\tau} : \mathcal{C} \longrightarrow \mathcal{C} 
\quad \text{and} \quad 
\tau^0 : \mathcal{C}^0 \longrightarrow \mathcal{C}^0.
\]
\end{theorem}

\subsection{Reduction}\label{reduction rp}

 The set of $\overline{K}$-points, $C(\overline{K})$, comes with a natural action of the absolute Galois group
$\text{Gal}(\overline{K}/K)$.
A closed point $ x \in C$  corresponds to a finite extension of  $K$, namely its residue field  $\kappa(x)$. If one chooses a geometric point  $\overline{x} \in C(\overline{K}) $ lying over  $x$, then the orbit of $\overline{x}$ under the Galois action,
\begin{equation*}
\operatorname{Orb}(\overline{x}) = \{ \sigma(\overline{x}) \mid \sigma \in \operatorname{Gal}(\overline{K}/K) \},
\end{equation*}
is in bijection with the closed point $x$. In other words, the closed point  $x$  can be identified with the $\operatorname{Gal}(\overline{K}/K)$-orbit of any of its geometric points.

This identification implies that studying the arithmetic of $ C$  via its closed points is equivalent to studying the Galois orbits in  $C(\overline{K})$. Moreover, the degree of the closed point  $x$, that is the extension degree $[ \kappa(x) : K]$, equals the cardinality of the orbit of any of its lifts in $C(\overline{K})$.

 \begin{definition}
Following \cite{hindry2013diophantine}, for a smooth and projective curve $C$, we construct a reduction map $r_{\pp}$ for any prime $\pp \notin S$. Let $x$ be a closed point of degree $m$, then define the reduction $r_{\pp}(x)$ as
\begin{equation*}
    r_\pp(x) = \overline{\{x\}} \cap C_{\pp}.
\end{equation*}
Since $x$ is a point of degree $m$, its closure $\overline{\{x\}}$ forms a horizontal divisor of degree $m$ on the arithmetic surface $\mathcal{C}$. The intersection of this divisor with the fiber $C_\pp$ yields a divisor of degree $m$ on $C_p$. Extending this process linearly to all closed points in $C$, we obtain a reduction map $r_{\pp}$ from the divisor group of $C$ to the divisor group of $C_\pp$, preserving both degree and effectivity. Thus, we have:
\begin{equation*}
    r_{\pp}:\Div^n_{+}(C) \longrightarrow \Div^n_{+}(C_{\pp}).
\end{equation*}
 \end{definition}
There is a left action of $\Aut_{\oks}(\CC)$ on $\Div(C)$, the divisor group of $C$, given by the pullback of divisors. For any $f \in \Aut_{\oks}(\CC)$ and $D \in \Div(C)$, this action is denoted:

\begin{equation*}
    f\cdot D  \vcentcolon= f_0^*D,
\end{equation*}
where $f_0$ is the induced automorphism on the generic fiber. This action preserves the degree of the divisors. Additionally, we have:
\begin{equation*}
    \Supp (r_\pp(f\cdot D)) = f_\pp^{-1}\big( \Supp (r_\pp(D))\big),
\end{equation*}
where $f_\pp$ is the induced automorphism on the fiber over $\pp$, so this action preserve reducedness as well. Thus, $\Aut_{\oks}(\CC)$ also acts on $ \Omega_{n,\ol K}(C;S)$ and  $ \Omega_{n, K}(C;S)$.

Let $L$ be a field extension of $K$ with ring of integers $\Oo_L$. Define $T(L)$ as the set of all primes in $\Oo_L$ lying above the primes in $S$. When the field $L$ is clear from context, we simply write $T$. Let $\qq$ be a prime in $\Oo_L$ lying above a prime $\mathfrak{p}$ in $\Oo_{K,S}$. Denote:

\begin{align*}
    C_L &= C \otimes_K L,      \\
     C_{\qq} &= \CC \otimes_{\F_{\pp}} \F_\qq,\\
    \CC_{L} &= \CC \otimes_{\Oo_{K,S}} \Oo_{L,T},\\
\end{align*}
where $\F_{\pp} = \oks/\pp$ and similarly $\F_{\qq} = \olt/\qq$.

\begin{lemma}There is a bijection between the set of $L$-rational points $C(L)$ and the set of $\olt$-integral points $\CC(\olt)$. Each $f \in \CC(\olt)$ corresponds uniquely to a point $f_0 \in C(L)$, and vice versa.
\end{lemma}

Consider $L$ a Galois field extension of $K$. Then we have the following commutative diagram:
\begin{center}
    \begin{tikzcd}
                                                                      & \Spec L \arrow[rr] \arrow[dd]                  &  & \Spec K \arrow[dd]                          \\
\Spec L \arrow[ru, "\sigma"] \arrow[rrru] \arrow[dd]                  &                                                &  &                                             \\
                                                                      & \Spec \, \Oo_{L,T} \arrow[rr]                          &  & \Spec \, \Oo_{K,S}                                  \\
\Spec \, \Oo_{L,T} \arrow[ru, "\sigma"] \arrow[rrru]                          &                                                &  &                                             \\
                                                                      & {\Spec \, \F_\qq} \arrow[rr] \arrow[uu, hook] &  & {\Spec \, \F_\mathfrak{p}} \arrow[uu, hook] \\
{\Spec \,\F_\qq} \arrow[ru, "\sigma"] \arrow[rrru] \arrow[uu, hook] &                                                &  &                                            
\end{tikzcd}
\end{center}
where $\sigma$ is an element of the decomposition group of $\qq$. By considering the fiber product under the map $\CC \to \Spec \, \mathcal{O}_{K,S}$, we obtain the following commutative diagram:

\begin{equation}\label{diag first}
 \begin{tikzcd}
                                                             &  & C_L \arrow[rrrr, "\pi"] \arrow[dd]          &  &  &  & C \arrow[dd]           \\
C_L \arrow[rru, "\sigma"] \arrow[rrrrrru] \arrow[dd]         &  &                                             &  &  &  &                        \\
                                                             &  & \CC_L \arrow[rrrr]                          &  &  &  & \CC                    \\
\CC_L \arrow[rru, "\sigma"] \arrow[rrrrrru]                  &  &                                             &  &  &  &                        \\
                                                             &  & C_\qq \arrow[rrrr, "\phi"] \arrow[uu, hook] &  &  &  & C_\pp \arrow[uu, hook] \\
C_\qq \arrow[rru, "\sigma"] \arrow[rrrrrru] \arrow[uu, hook] &  &                                             &  &  &  &                       
\end{tikzcd}
\end{equation}

\subsection{Cohomology}

For this section, we closely adhere to Chapter 27 of Milne's notes on algebraic groups, as presented in \cite{milne2014algebraic}. More detailed expositions can also be found in \cite{serre}.

\begin{definition}
    Let $G$ be a group. A $G$-group $A$ is a group $A$ with an action 
    \begin{equation*}
        (\sigma, a) \mapsto \sigma a : G \times A \rightarrow A
    \end{equation*}
of $G $ on the group $A$.
\end{definition}

 Let $A$ be a $G$-group. Then  $H^0(G,A) :=A^G$  the set of elements in $A$ fixed under the action of $G$, i.e.,
    \begin{equation*}
        H^0(G,A)=A^G=\{a\in A \mid \sigma a =a \,\,\,\, \forall \sigma \in G\}.
    \end{equation*}

\begin{definition}
    Let $A$ be a $G$-group. Define $Z(G,A)$   the sets of $1$-cocycles as follows:
    \begin{equation*}
        Z(G, A) = \{ f: G \rightarrow A \mid f(\sigma \tau ) = f(\sigma) \cdot \sigma f(\tau) \,\,\,\,\,\, \forall \sigma , \tau \in G \}.
    \end{equation*}    
\end{definition}
Two $1$-cocycles $f, g \in Z(G,A)$ are equivalent if there exists $c\in A$ such that 
 \begin{equation*}
     g(\sigma) = c^{-1}\cdot  f(\sigma)\cdot  \sigma c \,\,\,\,\,\,\,\,\,\,\,\,\,\forall \sigma \in G.
 \end{equation*}
This is an equivalence relation on the set of $1$-cocycles, and $H^1(G,A)$ is defined to be the set of equivalence classes of $1$-cocyces. In general $H^1(G,A)$  is not a group unless $A$ is commutative, but it has a distinguished element, namely, the class of $1$-cocycles of the form $\sigma \mapsto b^{-1}\cdot \sigma b, \, b\in A$ (the principal $1$-cocycle).

When $  A$ is commutative, $H^i(G,A)$ coincides with the usual cohomology groups for
$i=0, 1$.

\begin{theorem}
    An exact sequence 
    \begin{equation}\label{exact seq}
       \begin{tikzcd}
    1 \arrow[r] & A \arrow[r, "u"] & B \arrow[r, "v"] & C \arrow[r] & 1
\end{tikzcd}
    \end{equation}
    of $G$-groups gives rise to an exact sequence of pointed sets
    \begin{equation*}
        \begin{tikzcd}
1 \arrow[r] & {H^0(G,A)} \arrow[r, "u^0"] & {H^0(G,B)} \arrow[r, "v^0"] & {H^0(G,C)} \arrow[r, "\delta"] & {H^1(G,A)} \arrow[r, "u^1"] & {H^1(G,B)} \arrow[r, "v^1"] & {H^1(G,C)}.
\end{tikzcd}
    \end{equation*}
    More precisely:
    \begin{itemize}
        \item The sequence $  \begin{tikzcd}
1 \arrow[r] & {H^0(G,A)} \arrow[r, "u^0"] & {H^0(G,B)} \arrow[r, "v^0"] & {H^0(G,C)} 
\end{tikzcd}$ is exact as a sequence of groups.
\item There is a natural action of $C^G$ on $H^1(G,A)$.
\item The map $\delta $ sends $c\in C^G$ to $1\cdot c$ where $1$ is the distinguished element of $H^1(G,A)$ 
\item The nonempty fibres of $u^1:H^1(G,A) \rightarrow H^1(G,B)$ are the orbits of $C^G$ on $H^1(G,A)$.

\item The kernel of $v^1$ is the quotient of $H^1(G,A)$ by the action of $C^G$.
 
\end{itemize}
\end{theorem}

In non-commutative cohomology, the term "kernel" refers to the fiber over the distinguished element. This theorem describes only the fiber of $v^1$ that contains the class of principal $1$-cocycles. To describe the other fibers we need to consider appropriate twists of the $G$ action. 

\begin{definition}
    Let $B $ be a $G$-group, and let $S$ be a $G$-set with a left action of $B$ compatible with the action of $G$. Let $f\in Z(G,B)$, and let $\prescript{}{f}{S}$ denote the set $S$ on which $G$ acts by 
    \begin{equation*}
        \sigma * s = f(\sigma)\cdot \sigma s \quad\quad\forall \sigma \in G.
    \end{equation*}
 We say $\prescript{}{f}{S}$ is obtained from $S$ by twisting by a $1$-cocycle $f$.
\end{definition}

Now consider an exact sequence (\ref{exact seq}), and let $f\in Z(G,B)$. The group $B$ acts on itself by inner automorphism leaving $A$ stable, and so we can twist (\ref{exact seq}) by $f$ to obtain an exact sequence 
\begin{equation*}
    \begin{tikzcd}
1 \arrow[r] & \prescript{}{f}{A} \arrow[r, "u"] & \prescript{}{f}{B} \arrow[r, "v"] & \prescript{}{f}{C} \arrow[r] & 1.
\end{tikzcd}
\end{equation*}
The next theorem describes the fiber of $v^1$ containing  $[f]\in H^1(G,B)$.

\begin{theorem}\label{coh theo}
    There is a commutative diagram
    \begin{equation}
        \begin{tikzcd}
{H^0(G,\prescript{}{f}{C} )} \arrow[r] & {H^1(G,\prescript{}{f}{A} )} \arrow[r] & {H^1(G,\prescript{}{f}{B} )} \arrow[d, "\cong"] \arrow[r] & {H^1(G,\prescript{}{f}{C} )} \arrow[d, "\cong"] \\
{H^0(G,C)} \arrow[r, "\delta"]         & {H^1(G,A)} \arrow[r, "u^1"]            & {H^1(G,B)} \arrow[r, "v^1"]                               & {H^1(G,C)}                                     
\end{tikzcd}
    \end{equation}
in which the vertical arrows map the distinguished elements in $H^1(G,\prescript{}{f}{B})$ and $H^1(G,\prescript{}{f}{C})$
    to the class of $[f]$ and $v^1([f])$.
\end{theorem}

\begin{lemma}\label{lem:H1-finite}
Let $G$ be a finite group of order $n$, and let $A$ be a torsion‐free abelian group such that the quotient $A/nA$ is finite.  Then $H^1(G, A)$ is finite for any $G$-group structure on $A$.
\end{lemma}
\begin{proof}
Since $A$ is torsion‐free, multiplication by $n$
\begin{equation*}
\times n\colon A \longrightarrow A
\end{equation*}
is injective. Thus, we have a short exact sequence of $G$–groups
\begin{equation*}
1 \;\longrightarrow\; A 
\;\xrightarrow{\;\times n\;}\; A 
\;\longrightarrow\; A/nA 
\;\longrightarrow\; 1.
\end{equation*}
Taking cohomology, we obtain the associated long exact sequence
\begin{equation*}
\begin{aligned}
H^0(G,A)\;&\xrightarrow{\;\times n\;}\;H^0(G,A)\;\longrightarrow\;H^0(G,A/nA)\\
&\xrightarrow{\;\delta\;}H^1(G,A)
\xrightarrow{\;\times n\;}H^1(G,A)
\;\longrightarrow\;H^1(G,A/nA)\;\longrightarrow\;\cdots
\end{aligned}
\end{equation*}
Let $f\in Z^1(G,A)$ then for any  $\sigma \in G$   
\begin{equation*}
    f(\sigma)^n = \prod_{\tau \in G} f(\sigma \tau)\cdot \big(\sigma f(\tau)\big)^{-1} =\big[ \prod_{\tau\in G } f(\tau)^{-1} \big]^{-1} \cdot \sigma \big[ \prod_{\tau\in G } f(\tau)^{-1} \big].
\end{equation*}
Therefore, multiplication by $n$ on $H^1(G,A)$ vanishes. This implies that the map ${H^1(G, A)} \ra  {H^1(G, A/nA)} $ is injective. But $A/nA$ is finite, so $H^1(G,A/nA)$ is finite, and hence its subgroup $H^1(G, A)$ is also finite.
\end{proof}

\begin{lemma}\label{cohomology lemma}
Let $G$ be a finite group. Consider the following short exact sequence of $G$-groups.
\begin{equation*}
    \begin{tikzcd}
    1 \arrow[r] & A \arrow[r] & B \arrow[r] & C \arrow[r] & 1
    \end{tikzcd}.
    \end{equation*}
Let $A$ be a torsion free abelian group such that $A/\lvert G \rvert A$ is finite and $C$ be a $G$ group with finite $H^1(G, C)$. Then $H^1(G,B)$ is finite.
\end{lemma}
\begin{proof}
The short exact sequence induces a long exact sequence of pointed sets in nonabelian cohomology. Since $H^1(G, C)$ is finite and the first cohomology of any twist of $A$ is finite by Lemma \ref{lem:H1-finite}, the finiteness of $H^1(G, B)$ follows from Theorem~\ref{coh theo}.
\end{proof}

\begin{highlighted}
   Verify that the section title is appropriate from the perspective of algebraic geometry and accurately reflects the mathematical content and goals of this section.
\end{highlighted}
\section{Finiteness for $K$ points with standard model}\label{sec 4}
We adopt the setup introduced in the introduction. Let $K$ be a number field with ring of integers $\mathcal{O}_K$, and let $S$ be a finite set of primes of $\mathcal{O}_K$. Let $C$ be a smooth projective curve of genus $g$ defined over $K$, and let $\mathcal{C}$ be a smooth proper model of $C$ over $\mathcal{O}_{K,S}$, with automorphism group $\Aut_{\mathcal{O}_{K,S}}(\mathcal{C})$.

\begin{highlighted}
    This sentence needs to be checked.
\end{highlighted}
In this section, we prove the following theorem, which establishes the finiteness of orbits in all cases relevant to Theorems~\ref{K-points Theo}, except when $C \cong \mathbb{P}^1_K$ and the model $\mathcal{C}$ is not isomorphic to $\mathbb{P}^1_{\mathcal{O}_{K,S}}$.

\begin{highlighted}
   Mention the equivalent of theorem 1 that we can prove here without reaching the case with not $\P^1$ models. 
\end{highlighted}
\begin{theorem}\label{SplitTheoremAllButOne}
Let $C$ be a smooth projective curve over $K$ with a smooth and proper model $\mathcal{C}$ over $\mathcal{O}_{K,S}$. Assume that one of the following holds:
\begin{itemize}
    \item[\textnormal{(a)}] $C \cong \mathbb{P}^1_K$ and $\mathcal{C} \cong \mathbb{P}^1_{\mathcal{O}_{K,S}}$.
    \item[\textnormal{(b)}] $C$ has genus one.
    \item [\textnormal{(c)}]$C$ has genus at least two.
\end{itemize}
Then the set
\begin{equation*}
    \Aut_{\mathcal{O}_{K,S}}(\mathcal{C}) \backslash \Omega_{n,K}(C;S)
\end{equation*}
is finite.
\end{theorem}
We complete the proof of Theorem~\ref{SplitTheoremAllButOne} in the following subsections by treating the three cases separately.

\subsection{Genus 0 curves with standard model}

In this subsection, we prove Theorem~\ref{SplitTheoremAllButOne} in the case where the curve is $C \cong \mathbb{P}^1_K$ and the model is the standard one, namely $\mathbb{P}^1_{\mathcal{O}_{K,S}}$. We begin with the case $n=1$. Then $\Omega_{1,K}(\mathbb{P}^1;S)=\mathbb{P}^1(K)$ and $\Aut_{\mathcal{O}_{K,S}}(\mathcal{C})=\PGL_2(\mathcal{O}_{K,S})$. There is a natural bijection $\SL_2(\mathcal{O}_{K,S})\backslash \mathbb{P}^1(K)\cong \Cl_S(K)$, which sends a point $[\alpha:\beta]\in\mathbb{P}^1(K)$ to the ideal class of the fractional ideal $\alpha\mathcal{O}_{K,S}+\beta\mathcal{O}_{K,S}$ (see Theorem~4.2 in \cite{conradSL2}). This shows that the finiteness of the class group implies the desired finiteness.

We now assume that $n \geq 2$. The remaining cases, where the model is not isomorphic to $\mathbb{P}^1_{\mathcal{O}_{K,S}}$, will be treated in a later section using a descent argument and the results established here.

\begin{definition} 
For any commutative ring $R$, let $V_n(R)$ denote the space of homogeneous polynomials of degree $n$ in two variables with coefficients in $R$, known as binary $n$-ic forms. The group $\mathrm{GL}_2(R)$ acts on $V_n(R)$ by linear substitution: for a binary $n$-ic form $f(x, y) \in V_n(R)$ and $\gamma \in \GL_2(R)$,  
the action of $\gamma$ on $f$ is defined by  
\begin{equation*}  
\gamma \cdot f(x, y) \vcentcolon= f((x,y) \gamma). 
\end{equation*}  
\end{definition}  
  
\begin{definition} Suppose that $f(x, y) \in V_n(\oll)$ splits as \begin{equation*} f(x, y) = a_n \prod_{i=1}^n (\beta_i x - \alpha_i y), \end{equation*} where $[\alpha_i : \beta_i]$ are the roots of $f(x, y)$ in $\mathbb{P}^1_{\overline{L}}$. The \emph{discriminant} $\Delta(f)$ is defined by \begin{equation*} \Delta(f) = a_n^{2n-2} \prod_{1 \le i < j \le n} (\alpha_i \beta_j - \alpha_j \beta_i)^2. 
\end{equation*}  
It is known that $\Delta(f) = 0$ if and only if $f$ has a repeated root in its splitting field. Furthermore, $\Delta(f)$ is a relative invariant under the action of $\mathrm{GL}_2(\O_{L})$ meaning that
\begin{equation*}
    \Delta(\gamma\cdot f)= \det(\gamma)^{n(n-1)}\Delta(f),
\end{equation*}
for any $\gamma\in \GL_2(L)$.
\end{definition}

\vspace{.05in}
\noindent 
\textbf{Proof of Theorem~\ref{SplitTheoremAllButOne} in the case $C \cong \mathbb{P}^1_K$, $\mathcal{C} \cong \mathbb{P}^1_{\mathcal{O}_{K,S}}$, and $n=2$:} Let $A = \{a, a'\} \in \Omega_{2, K}(\mathbb{P}^1; S)$. Assume that the coordinates of $a, a' \in \mathbb{P}^1_{\mathcal{O}_{K,S}}$ are given by  
\begin{equation*}
    a = [\alpha:\beta], \quad a' = [\alpha':\beta'],
\end{equation*}
for some $\alpha, \beta, \alpha', \beta' \in \mathcal{O}_{K,S}$. We define a map  
\begin{equation*}
    \varphi: \Omega_{2, K}(\mathbb{P}^1; S) \to V_2(\mathcal{O}_{K,S}),
\end{equation*}
by associating $A$ with the binary quadratic polynomial  
\begin{equation*}
    \varphi_A(x,y) = \frac{1}{\alpha \beta' -  \alpha'\beta} (\beta x - \alpha y)(\beta' x - \alpha' y).
\end{equation*}
Since $A \in \Omega_{2, K}(\mathbb{P}^1; S)$, the coordinates satisfy the ideal relation  
\begin{equation*}
    (\alpha, \beta)(\alpha', \beta') = (\alpha \beta' - \alpha' \beta)
\end{equation*}
in $\mathcal{O}_{K,S}$. This ensures that $\varphi_A(x,y) \in \mathcal{O}_{K,S}[x,y]$ and has discriminant 1.

The space of binary quadratic forms is characterized by a single invariant: the discriminant. As a result, the number of $\SL_2(\mathcal{O}_{K,S})$-orbits of binary quadratic forms with a fixed nonzero discriminant is finite. Applying this to the case where the discriminant is $1$, we conclude that the set of binary quadratic forms in $V_2(\mathcal{O}_{K,S})$ with discriminant $1$ decomposes into finitely many $\SL_2(\mathcal{O}_{K,S})$-orbits. 

Since the map $\varphi$ is injective and respects the action of $\SL_2(\mathcal{O}_{K,S})$, it follows that $\Omega_{2, K}(\mathbb{P}^1; S)$ inherits the finiteness property and decomposes into finitely many orbits under the action of $\SL_2(\mathcal{O}_{K,S})$ and therefore of the action of $\PGL_2(\oks)$.\hfill$\qedsymbol$

\begin{definition}
    Two binary $n$-ic forms $f$ and $g$ in $V_n(\oll)$ are called $(L, T)$-equivalent, if there exists $\gamma \in \mathrm{GL}_2(\olt)$ and $\lambda \in \O_{L,T}^{\times}$ such that
    \begin{equation*}
        \gamma \cdot f = \lambda g.
    \end{equation*}
\end{definition}

\begin{theorem}[Birch-Merriman, \cite{birch1972finiteness}] \label{BM} Let $L$ be a number field, and let $T$ be a finite set of prime ideals in $\mathcal{O}_L$. For any $n \geq 3$, there are only finitely many $(L, T)$-orbits of binary $n$-ic forms $f \in V_n(\mathcal{O}_L)$ satisfying $\Delta(f) \in \mathcal{O}_{L,T}^{\times}$.
\end{theorem} 

\begin{lemma}
Let $H$ be the Hilbert class field of $K$. Then there exists an injective map
\begin{equation*} 
    \varphi: \Omega_{n,K}(C; S) \to  \O_{H}^{\times}\backslash V_n(\O_{H}), 
\end{equation*} 
where $\O_{H}^{\times}$ acts by scalar multiplication. The image of $\varphi$ consists of the set of binary $n$-ic forms with discriminants in $\O_{H,T}^{\times}$. 
\end{lemma}

\begin{proof}
Let $A = \big\{[\alpha_1 : \beta_1], \ldots, [\alpha_n : \beta_n]\big\}$ be an element of $\Omega_{n,K}(C; S)$, where $\alpha_i, \beta_i \in \ok$ for each $i$. We define $\varphi_A$ as the binary form  
\begin{equation*}
    \prod_{i} (\beta_i' x - \alpha_i' y),
\end{equation*}
where $[\alpha_i' : \beta_i'] \in \mathbb{P}^{1}_\oh$ are chosen such that $\gcd(\alpha_i', \beta_i') = 1$ in $\oh$ and satisfy
\begin{equation*}
    [\alpha_i : \beta_i] = [\alpha_i' : \beta_i']
\end{equation*}
in $\mathbb{P}^1_H$. Since $H$ is the Hilbert class field of $K$, every ideal of $\mathcal{O}_K$ is principal in $\mathcal{O}_H$, guaranteeing the existence of such coprime representatives. Consequently, $\varphi$ is well-defined up to $\mathcal{O}_H^{\times}$-equivalence: any other choice of $\alpha_i'$ and $\beta_i'$ results in a binary form differing only by a unit factor, which we quotient out by $\mathcal{O}_H^{\times}$. This also ensures that $\varphi$ is injective.

Now, we analyze the discriminant of $\varphi_A$. Let $\qq$ be a prime of $\oh$ lying above a prime $\pp$ of $\ok$, and suppose that $\qq \mid \Delta(\varphi_A)$. Then there exist indices $i$ and $j$ such that
\begin{equation}\label{qq-equation}
    \alpha_i'\beta_j' \equiv \alpha_j'\beta_i' \pmod{\qq}.
\end{equation}
We consider two cases:

\begin{itemize}
\item[\textnormal{(a)}] Both $\beta_i'$ and $\beta_j'$ are units in $(\oh)_\qq$:  
    Then we can rewrite equation \eqref{qq-equation} as  
    \begin{equation*} 
        \alpha_i' (\beta_i')^{-1} \equiv \alpha_j' (\beta_j')^{-1} \pmod{\qq}. 
    \end{equation*}  
    By construction, $\alpha_i' (\beta_i')^{-1} = \alpha_i \beta_i^{-1}$ and $\alpha_j' (\beta_j')^{-1} = \alpha_j \beta_j^{-1}$, meaning that these values are in $(\ok)_{\pp}$ and have the same reduction modulo $\pp$. Thus, $\pp \in S$ and $\qq \in T$.

\item[\textnormal{(b)}] $\beta_i'$ is divisible by $\qq$: 
    Then equation \eqref{qq-equation} implies that $\qq$ divides $\alpha_i' \beta_j'$. Since $\alpha_i'$ and $\beta_i'$ are coprime, it follows that $\qq$ divides $\beta_j'$. Consequently, both $\alpha_i'$ and $\alpha_j'$ are units in $(\O_{H})_{\qq}^{\times}$, and we rewrite equation \eqref{qq-equation} as  
    \begin{equation*}
        (\alpha_i')^{-1} \beta_i' \equiv (\alpha_j')^{-1} \beta_j' \pmod{\qq}.
    \end{equation*}
    Similarly, we obtain $(\alpha_i)^{-1} \beta_i \equiv (\alpha_j)^{-1} \beta_j \pmod{\pp}$, ensuring that $\pp \in S$ and $\qq \in T$.
\end{itemize}
Thus, the image of $\varphi$ is actually is set of binary forms with discriminant in $\O_{H,T}^{\times}$.
\end{proof}

\vspace{.1in}
\noindent 
\textbf{Proof of Theorem~\ref{SplitTheoremAllButOne} in the case $C \cong \mathbb{P}^1_K$, $\mathcal{C} \cong \mathbb{P}^1_{\mathcal{O}_{K,S}}$, and $n\geq3$:}  By Birch–Merriman \cite{birch1972finiteness}, there are only finitely many $(H, T)$-equivalence classes in $V_n(\oh)$ with discriminants in $\O_{H,T}^{\times}$. Consequently, we can select elements $A_1, \ldots, A_M$ in $\Omega_{n,K}(C; S)$ such that for any $A \in \Omega_{n,K}(C; S)$, there exist an index $i \in \{1, \ldots, M\}$ and $\gamma \in \mathrm{GL}_2(\O_{H,T})$ satisfying  

\begin{equation*} 
    c \, \varphi_{A} = \gamma \cdot \varphi_{A_i}, 
\end{equation*}  
where $c$ is a constant in $\O_{H,T}^{\times}$. By abuse of notation, we consider $\gamma$ as an automorphism of $\mathbb{P}^1_{\oht}$ that moves $A$ to one of the representatives $A_i$.  

To complete the proof, we now show that $\gamma$ is actually defined over $\oks$. Consider any $\sigma \in \Gal(H/K)$. Since $\sigma$ fixes every element of $\Omega_{n,K}(C; S)$, it follows that  
\begin{equation*} 
    \sigma^{-1} \gamma \sigma(A) = A_i.
\end{equation*}  
By the uniqueness property of automorphisms of $\mathbb{P}^1_{\oht}$ fixing three points, we deduce that  
\begin{equation*} 
    \sigma^{-1} \gamma \sigma = \gamma.
\end{equation*}  
Thus, by Galois descent, $\gamma$ descends to an automorphism of $\mathbb{P}^1_{\oks}$. This completes the proof. \hfill$\qedsymbol$

\subsection{Elliptic curves}
We now consider the case where $C$ is a genus 1 curve. Assuming that $\Omega_{n,K}(C; S)$ is nonempty, this implies that $C$ is isomorphic to an elliptic curve over $K$. In this case, the automorphism group of the model $\mathcal{C}$ over $\mathcal{O}_{K,S}$ is isomorphic to $\text{Aut}_K(C)$. Furthermore, the group $\text{Aut}_K(C)$ contains $C(K)$ via the action of translations. 

We prove the theorem by considering two cases: $n=1$ and $n\geq 2$.  

\vspace{.05in}
\noindent 
\textbf{Proof of Theorem \ref{SplitTheoremAllButOne} for the elliptic curves $n=1$:}  
In this case, we have  
\begin{equation*} 
\Omega_{1,K}(C; S) = C(K). 
\end{equation*}  
Since $C(K)$ modulo itself has only one equivalence class. Thus, the following set
\begin{equation*} 
C(K)\backslash \Omega_{1,K}(C; S) 
\end{equation*}  
consists of a single element, proving finiteness in this case. \hfill$\qedsymbol$ 

\vspace{.05in}
\noindent 
\textbf{Proof of Theorem \ref{SplitTheoremAllButOne} for the elliptic curves $n\geq 2$:} For any set $A = \{a_1, a_2, \ldots, a_n\} \in \Omega_{n,K}(C;S)$, we use the translation action of $C(K)$ to send $a_1$ to the identity element $O$ of $C$. This allows us to consider the translated set  
\begin{equation*} 
\{O, a_2 - a_1, \ldots, a_n - a_1\}. 
\end{equation*}  
Since $A$ satisfies the local conditions defining $\Omega_{n,K}(C;S)$, the elements $a_i - a_1$ for $i \geq 2$ must not be congruent to $O$ modulo any prime $\mathfrak{p} \notin S$. This implies that each $a_i - a_1$ is a nonzero $\mathcal{O}_{K,S}$-point on the elliptic curve $\CC \backslash \{O\}$.  

By Siegel's theorem on integral points, the set of $\mathcal{O}_{K,S}$-points on an elliptic curve is finite. Therefore, the differences $a_i - a_1$ can only take finitely many values. Thus, the quotient  
\begin{equation*} 
C(K)\backslash \Omega_{n,K}(C; S)
\end{equation*}  
is finite. This completes the proof of Theorem \ref{K-points Theo}. \hfill$\qedsymbol$

\subsection{Curves of genus at least two}  
By Faltings' theorem, the set of $K$-rational points $C(K)$ is finite when $C$ has genus at least $2$. Since $\Omega_{n, K}(C; S)$ consists of subsets of $C(K)$ satisfying local conditions, it follows that $\Omega_{n, K}(C; S)$ is finite. Consequently, the quotient  
\begin{equation*}  
\Aut_{\mathcal{O}_{K,S}}(\mathcal{C}) \backslash \Omega_{n, K}(C; S)  
\end{equation*}  
is also finite, completing the proof in this case. 

\section{Proof of Theorems}\label{sec 3}
To prove Theorem~\ref{Kbar-points Theo}, we first establish finiteness over a sufficiently large number field extension of $K$, and then descend to the base field. Specifically, we consider a finite extension $L/K$ over which all relevant points become $L$-rational points, allowing us to apply Theorem~\ref{SplitTheoremAllButOne} over $\mathcal{O}_{L,T}$ and deduce finiteness at that level.

More precisely, the natural base change induces an inclusion
\begin{equation*}
\iota: \Omega_{n,\overline{K}}(C;S) \hookrightarrow \Omega_{n,L}(C_L;T),
\end{equation*}
where $T$ is the set of primes of $\mathcal{O}_L$ lying above those in $S$. By choosing $L$ large enough, Theorem~\ref{SplitTheoremAllButOne} guarantees that the set $\Omega_{n,L}(C_L;T)$ breaks into finitely many orbits under the action of $\Aut_{\mathcal{O}_{L,T}}(\mathcal{C}_L)$.

To complete the proof, we descend this finiteness result to $\mathcal{O}_{K,S}$. Rather than analyzing the structure directly over $K$, we use the established finiteness at the level of $\mathcal{O}_{L,T}$. The key step is to show that each orbit at this larger level decomposes into finitely many orbits at the base level, ensuring finiteness over $K$. Specifically, for a fixed $A \in \Omega_{n,\overline{K}}(C;S)$, we establish that  
\begin{equation*}
\Aut_\olt(\CC_L) \cdot \iota(A) \cap \iota(\Omega_{n,\overline{K}}(C; S)),
\end{equation*}  
splits into finitely many orbits under $\iota(\Aut_{\oks}(\CC))$. This allows us to complete the proof of Theorems~\ref{K-points Theo} and \ref{Kbar-points Theo} by leveraging the finiteness result established in Theorem~\ref{SplitTheoremAllButOne}.

\subsection{Galois descent}
In this section, we establish the existence of a sufficiently large number field over which all points in our family are defined. We then relate the number of orbits in  
\begin{equation*}
\Aut_\oht(\CC_H) \cdot \iota(A) \cap \iota(\Omega_{n,\overline{K}}(C; S))
\end{equation*}  
to the first Galois cohomology group of the stabilizer of $\iota(A)$ in $\Aut_{\mathcal{O}_{H,T}}(\CC_H)$ for a suitably large number field $H$.

\begin{lemma}\label{discriminant-res}
There exists a finite Galois extension $L/K$ such that  
\begin{equation*}  
    \Omega_{n, \overline{K}}(C; S) \hookrightarrow \Omega_{n, L}(C_L; T),  
\end{equation*}
for the finite set of primes $T$ of $\mathcal{O}_L$ lying above $S$, and the map is induced by base change.  
\end{lemma}
\begin{proof}
Let $A \in \Omega_{n,\overline{K}}(C;S)$, and let $x \in A$. Then $x$ is defined over a finite extension $L_x = \kappa(x)$ of $K$. This closed point defines a horizontal divisor on the arithmetic surface $\mathcal{C} \to \operatorname{Spec} \mathcal{O}_{K,S}$, and its intersection with the fiber $\mathcal{C}_\pp$ over a prime $\pp \notin S$ gives a finite set of points in the special fiber.

Let $\pp_1, \dots, \pp_r$ be the primes of $\mathcal{O}_{L_x}$ lying above $\pp$, and let $f(\pp_i/\pp) := [\kappa(\pp_i) : \kappa(\pp)]$ denote the inertia degree. Then the number of geometric points in $\overline{\{x\}} \cap \mathcal{C}_\pp$ is bounded by
\begin{equation*}
\sum_{i=1}^r f(\pp_i / \pp).
\end{equation*}
Since $A \in \Omega_{n,\overline{K}}(C;S)$, so the closure of each $x \in A$ meets $\mathcal{C}_\pp$ in exactly $[L_x : K]$ geometric points. Therefore,
\begin{equation*}
[L_x : K] = \sum_{i=1}^r f(\pp_i / \pp),
\end{equation*}
which implies that $\pp$ is unramified in $L_x$. As this holds for all $\pp \notin S$, we conclude that $L_x/K$ is unramified outside $S$.

Moreover, since $A$ has $n$ elements, each point is defined over a field extension of degree at most $n$. Moreover, the set of possible fields $L_x$ with discriminant in $O_{K,S}^{\times}$ and degree less than $n$ is finite by Lemma 8 of~\cite{birch1972finiteness}. Let $L$ be the Galois closure of the compositum of all such fields $L_x$. Then $L/K$ is a finite Galois extension, and all points in $A$ are defined over $L$, so $A \in \Omega_{n,L}(C_L;T)$. This completes the proof of the lemma.
\end{proof}

For the remainder of this section, let $H/K$ be a finite Galois extension with Galois group $G = \Gal(H/K)$, chosen sufficiently large so that the natural map  
\begin{equation*} 
    \iota: \Omega_{n, \overline{K}}(C; S) \hookrightarrow \Omega_{n, H}(C_H; T),
\end{equation*}
is well-defined, where $\iota$ is induced by the base change $\Spec\, \mathcal{O}_{H,T} \to \Spec\, \mathcal{O}_{K,S}$. We additionally assume that $H$ contains the Hilbert class field of the field $L$ appearing in Lemma~\ref{discriminant-res}. The set $T$ denotes the set of primes of $\mathcal{O}_H$ lying above those in $S$.

\begin{definition}
Let $A \in \Omega_{n, \overline{K}}(C; S)$. Define the stabilizer of $A$ in the automorphism group of $\Aut_{\mathcal{O}_{H,T}}(\mathcal{C}_H)$:
\begin{equation*}
    M_A := \left\{ F \in \Aut_{\mathcal{O}_{H,T}}(\mathcal{C}_H) \mid F(A) = A \right\}.
\end{equation*}
Since $A$ is Galois-invariant, $M_A$ naturally carries an action of $G = \Gal(H/K)$, making it a $G$-group.
\end{definition}

\begin{proposition}\label{finite coh}
     Let $A\in \Omega_{n,\overline{K}}(C; S)$, then
    \begin{equation*}
        \#\,H^1(G,M_A)<\infty.
    \end{equation*}
\end{proposition}
\begin{proof}
We analyze the group $M_A \subset \Aut_{\mathcal{O}_{H,T}}(\mathcal{C}_H)$ and prove the finiteness of $H^1(G, M_A)$ case by case, according to the genus of the curve and the number of points in $A$.

\medskip
\noindent \textbf{Case 1:} Genus $g \geq 2$.\\
In this case, $\Aut_H(C_H)$ is finite, and hence so is $M_A \subset \Aut_H(C_H)$. It follows that $H^1(G, M_A)$ is finite.

\medskip
\noindent \textbf{Case 2:} Genus $g = 1$\\
In this case, the finiteness of the group $M_A$ can be established as follows. Since $A \subset C(H)$ contains finitely many points, the subgroup of translations preserving $A$, namely $C(H) \cap M_A$, is finite. Moreover, the quotient $M_A / (C(H) \cap M_A)$ embeds into the quotient $\Aut_H(C_H) / C(H) \cong \Aut(C_H, O)$, which is also finite. It follows that $M_A$ is an extension of a finite group by a finite group, and is therefore itself finite. Consequently, the cohomology group $H^1(G, M_A)$ is finite.

\medskip
\noindent \textbf{Case 3:} Genus $g = 0$.

\begin{itemize}
\item[\textbf{(a)}] \textbf{$n \geq 3$:} \\
Any automorphism of $\mathbb{P}^1$ that fixes three or more points must be the identity. Hence $M_A$ is a finite group and $H^1(G, M_A)$ is finite.

\item[\textbf{(b)}] \textbf{$n = 2$:} \\
In this case, $M_A$ fits into a short exact sequence of $G$-groups:
\begin{equation*}
\begin{tikzcd}
1 \arrow[r] & \mathcal{O}_{H,T}^\times \arrow[r] & M_A \arrow[r] & \mathbb{Z}/2\mathbb{Z} \arrow[r] & 1.
\end{tikzcd}
\end{equation*}
Since  $\mathcal{O}_{H,T}^\times$ is a finitely generated abelian group,
\cite[Corollary 2.1.32]{milne2011class}
implies that $H^1(G, M_A)$ is finite for all the twists of the action of $G$ on $\mathcal{O}_{H,T}^\times$. The group $\Z/2\Z$ is a finite and therefore $H^1(G,M_A)$ is finite by Theorem \ref{coh theo}.

\item[\textbf{(c)}] \textbf{$n = 1$:} \\
A point in $\mathbb{P}^1$ corresponds to a Borel subgroup of $\PGL_2(\mathcal{O}_{H,T})$, and such a group fits into a short exact sequence
\begin{equation*}
\begin{tikzcd}
1 \arrow[r] & \mathcal{O}_{H,T} \arrow[r] & B \arrow[r] & \mathcal{O}_{H,T}^\times \arrow[r] & 1,
\end{tikzcd}
\end{equation*}
where $B \cong M_A$ is the stabilizer of a point in $\PGL_2(\mathcal{O}_{H,T})$. 

The cohomology group of $H^1(G,\mathcal{O}_{H,T}^\times )$ is finite and $\oht $ is a torsion free abelian group such that $\oht/|G|\oht $ is finite. Thus, Lemma \ref{cohomology lemma} implies that $H^1(G,
M_A)$ is finite.
\end{itemize}
In all cases, we conclude that $H^1(G, M_A)$ is finite.
\end{proof}

\begin{lemma}\label{embed to cohomo}
Let $A \in \Omega_{n, \overline{K}}(C; S)$, and let $H/K$ be as before. Then, there exists an injective map  
    \begin{equation}
        \Aut_\oks(\CC) \backslash \Sigma_A \xhookrightarrow{\quad} H^1(G, M_A),  
    \end{equation}
where $\Sigma_A$ is defined by
\begin{equation*}
    \Sigma_A = \Aut_\oht(\CC_H) \cdot \iota(A) \cap 
    \iota(\Omega_{n,\overline{K}}(C; S)).
\end{equation*}
\end{lemma}

\begin{proof}
Let $B \in \Sigma_A$ and assume that $B = F \cdot A$ for some $F \in \Aut_\oht(\CC_H)$. Since both $A$ and $B$ are Galois-invariant sets of points, we have:
\begin{equation*}
    \begin{tikzcd}
        B \arrow[d, "\sigma"] &  & A \arrow[ll, "F"'] \arrow[d, "\sigma"] \\
        \sigma B = B            &  & \sigma A = A \arrow[ll, "\sigma F"']    
    \end{tikzcd}
\end{equation*}
where $\sigma F$ denotes the $\sigma F \sigma^{-1}$. From this diagram, we define a $1$-cocycle:  

\begin{align*}
    \psi_{B,F} :\, &G \to M_A \\
    &\sigma \mapsto F^{-1} \cdot \sigma F.
\end{align*}
If $B \in \Sigma_A$ admits two different factorizations for some $ F_1, F_2 \in \Aut_\oht(\CC)$
\begin{equation*}
    B = F_2 \cdot A = F_1 \cdot A,
\end{equation*}  
then  $F_1^{-1} \cdot F_2 \in M_A$. Using this, we compute  
\begin{equation*}
    \psi_{B,F_2}(\sigma)
    = (F_1^{-1} \cdot F_2)^{-1} \cdot \psi_{B,F_1}(\sigma) \cdot \sigma (F_1^{-1} \cdot F_2).
\end{equation*}
This shows that the classes of $\psi_{B,F_1}$ and $\psi_{B,F_2}$ are the same in $H^1(G,M)$. Thus, the definition depends only on $B$ and not on the choice of $F$, which implies that:
\begin{equation*}
    \psi :\, \Sigma_A \to H^1(G,M), \quad
    B \mapsto \psi_B.
\end{equation*}
Now, suppose that $B, C \in \Sigma_A$ and that $C = R \cdot B$ for some $R \in \Aut_\oks(\CC)$. Writing $B = F \cdot A$, we compute  

\begin{align*}
    \psi_C(\sigma) 
    &= ({R\cdot F})^{-1} \cdot \sigma (R\cdot F)  \\
    &= F^{-1} \cdot R^{-1} \cdot \sigma R \cdot \sigma F  \\
    &= F^{-1} \cdot R^{-1} \cdot R \cdot \sigma F  \\
    &= F^{-1} \cdot \sigma F \\
    &= \psi_B(\sigma).
\end{align*}
Thus, the map $\psi$ is invariant under the action of $\Aut_{\oks}(\CC)$, inducing a well-defined map  

\begin{equation*}
    \Aut_\oks(\CC) \backslash \Sigma_A \to H^1(G, M_A).
\end{equation*}
To show injectivity, assume that $B = F_1 \cdot A$ and $C = F_2 \cdot A$ satisfy $\psi_B = \psi_C$. Then, there exists an element $N \in M_A$ such that for all $\sigma \in G$,

\begin{align*}
    F_1^{-1} \cdot \sigma F_1 &= N^{-1} \cdot F_2^{-1} \cdot \sigma F_2 \cdot \sigma N.
\end{align*}
Rearranging, we obtain  

\begin{align*}
    F_2 \cdot N \cdot F_1^{-1} &= \sigma (F_2 \cdot N \cdot F_1^{-1}) \quad \forall \sigma \in G.
\end{align*}
Since this element is fixed by $G$, it must belong to $\Aut_{\oks}(\CC)$. Moreover, since  

\begin{equation*}
    (F_2 \cdot N \cdot F_1^{-1}) \cdot B = C,
\end{equation*}
we conclude that $\psi$ induces an injective map  
\begin{equation*}
    \Aut_\oks(\CC) \backslash \Sigma_A \xhookrightarrow{\quad} H^1(G, M).
\end{equation*}
 
\end{proof}

\begin{remark}
Lemma~\ref{embed to cohomo} also follows from the standard classification of twists via Galois cohomology.
\end{remark}

\subsection{Proof of Theorem \ref{K-points Theo}}
In this subsection we complete the proof of Theorem \ref{K-points Theo}. The only case we need to consider is $C \cong \P^1_K$ and the model $\CC$ over $\O_{K,S}$ not necessarily $\P^1_{\O_{K,S}}$. Then Assume $H$ as defined before, we have the following embedding:
\begin{equation*}
    \Omega_{n, \overline{K}}(C; S) \hookrightarrow \Omega_{n, H}(C_H; T).
\end{equation*}
Additionally, we know that the class group of $K$ becomes trivial in $H$, this implies that the model $\CC_H$ will be isomorphic to $\P^1_{\O_{H,T}}$. Thus, we are in the set up of the Theorem \ref{SplitTheoremAllButOne}. 

By Lemma \ref{embed to cohomo}
and Proposition \ref{finite coh}, we know each orbit over $\oht$ breaks into finitely many orbits over $\O_{K,S}$, which completes the proof for this case.\hfill$\qedsymbol$

\subsection{Proof of Theorem \ref{Kbar-points Theo}}
In this subsection, we complete the proof of Theorem~\ref{Kbar-points Theo}. Let $H/K$ be the Galois extension defined in the previous section. As established, we have a natural base change map
\begin{equation*}
    \Omega_{n, \overline{K}}(C; S) \hookrightarrow \Omega_{n, H}(C_H; T).
\end{equation*}
The base change induces a map on orbits:
\begin{equation*}
    \Aut_{\mathcal{O}_{K,S}}(\mathcal{C}) \backslash \Omega_{n, \overline{K}}(C; S) \longrightarrow \Aut_{\mathcal{O}_{H,T}}(\mathcal{C}_H) \backslash \Omega_{n, H}(C_H; T).
\end{equation*}
Since $\Aut_{\mathcal{O}_{K,S}}(\mathcal{C})$ is a subgroup of $\Aut_{\mathcal{O}_{H,T}}  (\mathcal{C}_H)$, this map is well-defined. By Theorem~\ref{K-points Theo}, the right-hand side is finite:
\begin{equation*}
    \#\left( \Aut_{\mathcal{O}_{H,T}}(\mathcal{C}_H) \backslash \Omega_{n, H}(C_H; T) \right) < \infty.
\end{equation*}
To conclude the proof, it suffices to show that each fiber of this map is finite. Let
\begin{equation*}
    \mathfrak{S} \in \Aut_{\mathcal{O}_{H,T}}(\mathcal{C}_H) \backslash \Omega_{n, H}(C_H; T)
\end{equation*}
be an orbit, and assume that $A \in \Omega_{n, \overline{K}}(C; S)$ maps to $\mathfrak{S}$. Then the fiber of $\mathfrak{S}$ is given by
\begin{equation*}
    \Aut_{\mathcal{O}_{K,S}}(\mathcal{C}) \backslash \Sigma_A,
\end{equation*}
where $\Sigma_A$ is the $\Aut_{\mathcal{O}_{H,T}}(\mathcal{C}_H)$-orbit of $A$. By Lemma~\ref{embed to cohomo}, there is an injective map
\begin{equation*}    \Aut_{\mathcal{O}_{K,S}}(\mathcal{C}) \backslash \Sigma_A \hookrightarrow H^1(G, M_A),
\end{equation*}
where $G = \Gal(H/K)$ and $M_A$ is the stabilizer of $A$ as a $G$-group.

Proposition~\ref{finite coh} ensures that $H^1(G, M_A)$ is finite, so $\Aut_{\mathcal{O}_{K,S}}(\mathcal{C}) \backslash \Sigma_A$ is finite as well. This establishes the finiteness of the fiber of $\mathfrak{S}$, completing the proof. \hfill$\qedsymbol$

 \bibliographystyle{amsalpha}
 \bibliography{Sigma}
\end{document}